\documentclass[11pt]{article}%
\usepackage{graphicx} 
\usepackage[latin1]{inputenc}
\usepackage[T1]{fontenc}
\usepackage[english]{babel}
\usepackage{geometry}
\usepackage{textcomp}
\usepackage{amsmath}
\usepackage{amsfonts}
\usepackage{amsthm}
\usepackage{amssymb}
\usepackage{mathrsfs}
\usepackage{fancyhdr}
\usepackage{indentfirst}

\usepackage{hyperref}
\usepackage{tikz}
\usepackage{tikzscale}
\usepackage[all]{xy}
\usepackage{xcolor}
\usepackage{setspace}
\usepackage{multicol}

\newtheorem{teo}{Theorem}
\newtheorem{cor}[teo]{Corollary}
\newtheorem{lem}[teo]{Lemma}
\newtheorem{prop}[teo]{Proposition}

\theoremstyle{definition}
\newtheorem{defn}[teo]{Definition}

\newtheorem{exem}{Example}
\newtheorem{obs}[teo]{Remark}
\newtheorem{theoremx}{Theorem}

\newcommand{\C}{\mathbb C}
\newcommand{\R}{\mathbb R}

\newcommand{\Z}{\mathbb Z}
\newcommand{\N}{\mathbb N}

\newcommand{\A}{\mathscr A}
\newcommand{\KS}{K(\mathcal{S})}
\newcommand{\sgs}{\mathcal{S}}

\newcommand{\dsp}{\displaystyle}

\newcommand{\xs}{X_{\mathcal{S}}}

\newcommand{\dku}{\Delta^1_k}
\newcommand{\dkd}{\Delta^2_k}

\newcommand{\cone}{\operatorname{Cone}}
\newcommand{\ks}{K(\mathcal{S})}
\newcommand{\ms}{\mathcal{S}}

\newcommand {\purple}[1]{{\color[rgb]{0.8,0.2,0.5} {#1}}}

\title{The Euler characteristic of Milnor fibers over 2-generic symmetric determinantal varieties}

\author{Tha\'is M. Dalbelo\footnote{Research supported by FAPESP-Grant 2019/21181-0, FAPESP-Grant 2024/22060-0 and by CNPq grant 403959/2023-3.}, Daniel Duarte\footnote{Research supported by SECIHTI project CF-2023-G-33.}, Danilo da N\'obrega Santos\footnote{Research supported by Edital PDSE no. 06/2024 and grant 88881.980859/2024-01.}}

\begin{document}

\maketitle

\begin{abstract}
In this work we present a formula for the Euler characteristic of the Milnor fiber of non-degenerate functions $f: X \to \mathbb{C}$ with isolated critical set relative to a stratification, where $X$ is a $2$-generic symmetric determinantal variety. The formula is obtained in two steps. Firstly, we explicitly describe the toric structure of those varieties. Secondly, we compute volumes of Newton polyhedra arising from the toric structure. The result then follows from Matsui-Takeuchi's formula for Milnor fibers over toric varieties. As an application, we compute the local Euler obstruction of $X$ at the origin and the local Euler obstruction of $f$. We also relate the Euler obstruction of $f$ to the Milnor number of a certain polynomial associated to $f$.
\end{abstract}

%\ccode{Mathematics Subject Classification 2000: 55S35, 32S05}

%\tableofcontents

\section*{Introduction}
The Milnor fiber of a function is a powerful tool in the analysis of critical points and has a wide range of applications in various areas of mathematics. Its importance lies in the fact that it provides a detailed local description of the behavior of functions around critical points, enabling significant advances in the understanding and classification of such critical points \cite{Milnor}. 

Let $X$ be a subvariety of $\mathbb{C}^N$ and $f: X \to \mathbb{C}$ a non-constant regular function. It is well known that, on appropriate neighborhoods, $f$ defines a topological fiber bundle \cite{MR1404920}. The Milnor fiber of $f$ at 0, denoted by $F_0$, is any fiber of this fibration. 

In this work we compute the Euler characteristic of the  Milnor fiber of non-degenerate functions $f$ with isolated critical set, when $X$ is a $2$-generic symmetric determinantal variety. Recall that generic symmetric determinantal varieties are defined in terms of vanishing of minors of symmetric matrices (by $2$-generic symmetric determinantal varieties we mean those defined by the vanishing of $2 \times2$ minors). This special class of varieties have been extensively studied (see, for instance, \cite{Beelen, symmetric, Gaube, Harris, Kotzev}). Notice that $2$-generic symmetric determinantal varieties are also toric varieties. Hence, it is natural to ask about the semigroup that defines them. Our first result provides an explicit description of such semigroups.

\begin{theoremx}[see Theorem \ref{esttoric}]\label{theorem 1}
Let $n\in\N$, $n\geq2$. Denote as $S_n^2$ the corresponding $2$-generic symmetric determinantal variety. Then $S_n^2$ equals the toric variety defined by the semigroup in $\Z^n$ generated by the following set
	\begin{align*}
		\mathscr{A}=\{& e_1, e_1+e_2, e_1+e_3,\dots, e_1+e_{n-1}, e_1+e_n, \\
		& e_1+2e_2, e_1+e_2+e_3, e_1+e_2+e_4,\dots, e_1+e_2+e_{n-1},e_1+e_2+e_n,\\
		& e_1+2e_3, e_1+e_3+e_4,\dots, e_1+e_3+e_{n-1}, e_1+e_3+e_n,\\
		&\qquad \vdots\\
		& e_1+2e_{n-1},e_1+e_{n-1}+e_n,\\
		& e_1+2e_n\}.
	\end{align*}
\end{theoremx}

With the explicit set of generators $\A$ at hand, we are able to deduce several  properties of the semigroup $\N\A$. Moreover, we fully describe the faces of the cone generated by $\A$. As a surprising fact, we show that this cone exhibits a kind of \textit{fractal behaviour}: each of its faces also determines a $2$-generic symmetric determinantal variety (see Proposition \ref{propface}). These results are key to derive the claimed formula  for the Euler characteristic. 

The other main tool for that goal is the work of Matsui-Takeuchi \cite{matsui2011}. In \textit{op. cit.}, the authors provide a formula for the Euler characteristic of Milnor fibers on toric varieties in terms of volumes of Newton polyhedra. We are able to explicitly compute every element of Matsui-Takeuchi's formula in our context, leading to the following result.

\begin{theoremx}[see Theorem \ref{differentds}]\label{theorem 2}	Let $S^2_n\subset\mathbb{C}^N$ be a 2-generic symmetric determinantal variety. Let $f:S^2_n\longrightarrow\mathbb{C}$ be a non-degenerate polynomial function given by
%\red{with an isolated critical point at the origin}  
	\[
	f(z)=\sum_{i=1}^{n}\alpha_i z_i^{d_i}+h(z),
	\]
where $\alpha_i \neq 0$ and $d_i\geq1$ for all $1\leq i \leq n$, and $h(z)$ is a polynomial function satisfying $\Gamma_+(h) \subset \Gamma_+(\sum_{i=1}^{n}z_i^{d_i})$ and $\Gamma_+(h)$ denotes the Newton polyhedron of $h$. Then the Euler characteristic of the Milnor fiber $F_0$ of $f$ in $0\in S_n^2$ is given by
	\[
	\chi(F_0)=\sum_{k=1}^n(-1)^{k-1}2^{k-1}\sum_{1\leq i_1<\dots <i_k\leq n}^{}d_{i_1} \dots d_{i_k}.
	\]
In particular, this formula holds for non-degenerate polynomial functions on $S^2_n$ with isolated critical point at the origin relative to a stratification.
\end{theoremx}

With this formula at hand, we derive the local Euler obstruction of $S^2_n$ and the local Euler obstruction of $f$. This provides a different proof for the computation of the Euler obstruction of $S^2_n$ already obtained in \cite{Raicu}. Finally, we provide a formula relating the Euler obstruction of $f$ to the Milnor number of a certain polynomial associated with $f$.

The paper is divided as follows. In the first section we collect the basics of toric geometry that we need as well as a discussion on Matsui-Takeuchi's formula. In Section 2 we prove Theorem \ref{theorem 1} and deduce several consequences of the semigroup defining $S^2_n$ and also of the cone it generates. Section 3 is devoted to proving Theorem \ref{theorem 2}. In the last section we compute the local Euler obstruction of $S_n^2$ and the local Euler obstruction of the function $f$ appearing in Theorem \ref{theorem 2}. We conclude by showing a relation among the local Euler obstruction of $f$ and the Milnor number of a function associated with $f$.

\section{A formula for the Euler characteristic on toric varieties}

We start by recalling the basics of toric geometry, followed by a discussion on a formula established by Matsui and Takeuchi for the Euler characteristic of the Milnor fiber of non-degenerate polynomial functions on toric varieties.

\subsection{Basic notions of toric geometry}

Let us recall the definition of a toric variety. A general treatment on toric geometry can be consulted in \cite{1996grobner,cox2011}.

\begin{defn}\label{toricvar}
Let $\mathcal{S}=\mathbb{N}\mathscr{A}\subset\mathbb{Z}^d$ be a semigroup finitely generated by $\mathscr{A}=\{m_1,\dots,m_s\}$. Consider the $\mathbb{C}$-algebra homomorphism
	\[
	\begin{array}{rccc}
		\pi_\mathcal{S} : & \mathbb{C}[z_1,\dots,z_s] & \longrightarrow & \mathbb{C}[x_1^\pm,\dots,x_d^\pm]\\
		&           z_i             &   \longmapsto & x^{m_i} \\
	\end{array}.
	\]
	Denote $I_\mathcal{S}=\mbox{ker}(\pi_\mathcal{S})$. The variety $X_\mathcal{S}=V(I_\mathcal{S})\subset\C^s$ is called a toric variety defined by $\mathcal{S}$. We denote as $\C[x^{\mathcal{S}}]$ the image of $\pi_{\mathcal{S}}$, which is the coordinate ring of $X_{\mathcal{S}}$. Moreover, we denote as $\C[\ms]$ the semigroup algebra generated by $\ms$. Notice that $\C[x^{\mathcal{S}}]\cong\C[\ms]$.
\end{defn}

Notice that we do not require a toric variety to be normal. It is a basic fact that $\xs$ is an irreducible variety having a dense open set isomorphic to an algebraic torus, whose action on itself extends to all of the variety. Its dimension is $\operatorname{rank}(\Z\A)$.

Normal toric varieties can also be characterized using rational polyhedral cones in $\R^d$, which are sets of the form
$$\cone(A)=\Big\{\sum_{u\in A}\lambda_uu\mid \lambda_u\in\R_{\geq0}\Big\},$$
where $A\subset\Z^d$ is finite. We refer to $\cone(A)$ simply as a cone. The dimension of a cone is the dimension of the real vector space it generates.

\begin{prop}\label{proptor}
Let $\mathcal{S}=\mathbb{N}\mathscr{A}\subset\mathbb{Z}^d$ be a semigroup finitely generated by $\mathscr{A}=\{m_1,\dots,m_s\}$. Then $X_\mathcal{S}$ is a normal variety if and only if $\mathcal{S}=\cone({\mathscr{A})}\cap\mathbb{Z}^d$.
\end{prop}

Given a cone $\sigma\subset\R^d$, its dual cone is defined as
$$\sigma^\vee=\{v\in\R^d \mid u\cdot v\geq0, \mbox{ for all }u\in\sigma\},$$
where $\cdot$ denotes the usual dot product in $\R^d$. It is also a cone.

A face of a cone $\sigma$ is any subset of the form $\sigma\cap H_v$, for some element $v\in\sigma^\vee$, where $H_v=\{u\in\R^d \mid u\cdot v=0\}$. A cone $\sigma$ is said to be strongly convex if $\{0\}$ is a face of $\sigma$.

\subsection{Matsui-Takeuchi's formula for the Euler characteristic}\label{MT formula}

Once we have established the basic notions on toric geometry that we need, we can now discuss a formula, established by Matsui and Takeuchi, for computing an Euler characteristic of Milnor fibers in the context of toric varieties.

Let $\mathcal{S}=\mathbb{N}\mathscr{A}\subset\mathbb{Z}^d$ be a semigroup finitely generated by $\mathscr{A}=\{m_1,\dots,m_s\}$. Denote as $\ks=\cone(\A)$. Recall that a polynomial function $f: X_{\mathcal{S}}\longrightarrow\C$ determines and is determined by an element of the semigroup algebra $\C[\ms]$, which is a sum of the form
$$f=\dsp\sum_{u\in\mathcal{S}}a_u\cdot u,\;\mbox{where}\; a_u\in\mathbb{C}.$$ 

%The following notions are constantly used from now on.

\begin{defn}\cite{matsui2011}\label{newpoly}
	Let $f=\dsp\sum_{u\in\mathcal{S}}a_u\cdot u$ be a polynomial function on $X_\mathcal{S}$. %We define:
	\begin{enumerate}
		\item[(i)] The support of $f$ is
		\[
		\mbox{supp}(f):=\{u\in\mathcal{S}: a_u\neq0 \}\subset\mathcal{S}.
		\]
		\item[(ii)] The Newton polyhedral $\Gamma_+(f)$ of $f$ is 
		\[
		\operatorname{Conv}\Big(\dsp\bigcup_{u\in\;\operatorname{supp}(f)}(u+K(\mathcal{S}))\Big),
		\] 
		 where $\operatorname{Conv}(\cdot)$ denotes the convex hull. Notice that $\Gamma_+(f)\subset\KS$.
	\end{enumerate}
\end{defn}

\begin{exem}
Let $\mathcal{S}$ be generated by $\mathscr{A}=\{e_1,e_1+2e_2,e_1+e_2\}\subset\mathbb{Z}^2$. Consider the polynomial function $f(z_1,z_2,z_3)=z_1^3+2z_2^3-z_3^3+4z_1^3z_3$. Thus, the corresponding element in the semigroup algebra $\mathbb{C}[\mathcal{S}]$ is 
	\[
	f=1\cdot(3e_1)+2\cdot(3e_1+6e_2)+(-1)\cdot(3e_1+3e_2)+4\cdot(4e_1+e_2),
	\]
	with $\mbox{supp}(f)=\{(3,0),(3,6),(3,3),(4,1)\}$ (see Figure \ref{figpoli}). 
\end{exem}

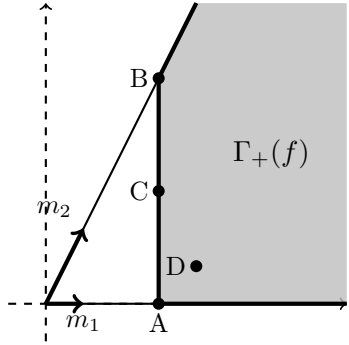
\begin{figure}[h]
	\centering
	\begin{tikzpicture}[scale=1]
	
	%EIXOS
			\draw[thick,dashed,->] (-0.5,0) -- (4,0);\draw[thick,dashed,->] (0,-0.5) -- (0,4);
	
	%GERADORES
			\draw[ultra thick,->] (0,0) -- (0.5,1)node[above left]{{\small$m_2$}};\draw[ultra thick,->] (0,0) -- (0.5,0)node[below]{{\small$m_1$}};

	%PONTOS
			\fill[lightgray, opacity=0.8] (4,0)--(1.5,0)--(1.5,3)--(2,4)--(4,4)--cycle;\fill[black] (1.5,0)node[below]{{\small A}} circle (0.8mm) ;\fill[black] (1.5,1.5)node[left]{{\small C}} circle (0.8mm) ;\fill[black] (1.5,3)node[left]{{\small B}} circle (0.8mm) ;\fill[black] (2,0.5)node[left]{{\small D}} circle (0.8mm) ; 
			
			\draw (3,2)node{$\Gamma_+(f)$};
	
	%CONE
			\draw[thick] (0,0) -- (2,4);\draw[thick] (0,0) -- (4,0);

	%POLIEDRO
			\draw[ultra thick]  (1.5,0) -- (4,0);\draw[ultra thick] (1.5,0) -- (1.5,3);\draw[ultra thick] (1.5,3) -- (2,4);

	\end{tikzpicture}
	\caption{Newton polyhedron of $f$.}
	\label{figpoli}
\end{figure}

For the remainder of this section we consider $f\in\mathbb{C}[\mathcal{S}]$ to be a polynomial function on $X_\mathcal{S}$ such that $0\not\in\mbox{supp}(f)$. To each $f=\sum_{u\in\mathcal{S}}a_u\cdot u\in\C[\sgs]$ we associate a Laurent polynomial 
\[
L(g)(x)=\sum_{u\in\mathcal{S}}a_ux^u
\]  
defined on $(\mathbb{C}^*)^d$. 

\begin{defn}\label{nondeg}
	Let $f=\sum_{u\in\mathcal{S}}a_u\cdot u\in\mathbb{C}[\mathcal{S}]$. We say that $f$ is non-degenerate if, for every compact face $\gamma\subset\Gamma_+(f)$, the hypersurface $L(f_\gamma)^{-1}(0)\subset(\mathbb{C}^*)^d$ is smooth and reduced, where $f_\gamma=\sum_{u\in\gamma\cap\mathcal{S}}a_u\cdot u$.
\end{defn}

In \cite{matsui2011} the authors provided a formula for computing the Euler characteristic of the Milnor fiber at $0\in X_\mathcal{S}$ for a non-degenerate polynomial function. In order to describe their formula some notation is required.

Let $f\in\C[\sgs]$. For each $k$-dimensional face $\Delta$ of the cone $\KS$ with $\Gamma_+(f)\cap\Delta\neq\emptyset$, denote by $\gamma^\Delta_1,\dots,\gamma^\Delta_{v(\Delta)}$ the compact faces of $\Gamma_+(f)\cap\Delta$ of dimension $k-1$.
%such that $\dim \gamma^\Delta_i=k-1$, $i\in\{1,\dots,v(\Delta)\}$. 
Let $L(\Delta)$ be the linear subspace of $\mathbb{R}^n$ generated by $\Delta$ and let $\Gamma^\Delta_i$ be the convex hull of $\gamma^\Delta_i\sqcup\{0\}$ in $L(\Delta)$. 

\begin{teo}\cite[Corollary 3.5]{matsui2011}\label{matsui2011}
	Let $f=\dsp\sum_{u\in\mathcal{S}}a_u\cdot u\in\mathbb{C}[\mathcal{S}]$ be a non-degenerate polynomial function on $X_\mathcal{S}$. Then the Euler characteristic of the Milnor fiber $F_0$ of $f$ in $0\in X_\mathcal{S}$ is given by
	\begin{equation}\label{formjapa}
		\chi(F_0)= \sum_{\Gamma_+(f)\cap\Delta\neq\emptyset}(-1)^{\dim\Delta-1}\sum_{i=1}^{v(\Delta)}\mbox{Vol}_\mathbb{Z}(\Gamma^\Delta_i)
	\end{equation}
	where $\mbox{Vol}_\mathbb{Z}(\Gamma^\Delta_i)$ is the normalized $(\dim\Delta)$-dimensional volume of $\Gamma^\Delta_i$ with respect to the lattice $\Z(\mathcal{S}\cap\Delta)$. 
\end{teo}

Recall that $\mbox{Vol}_\mathbb{Z}(\cdot)$ can be computed as follows. Let $\Delta$ be a face of the cone $\KS$, and let $L(\Delta)$ be the linear subspace it generates. Then, $\Z(\mathcal{S}\cap\Delta)\subset L(\Delta)$ is a $\mathbb{Z}$-lattice whose rank is $\dim L(\Delta)=\dim\Delta$. We have $\R(\mathcal{S}\cap\Delta)=L(\Delta)$. Fix a $\mathbb{Z}$-basis of $\Z(\mathcal{S}\cap\Delta)$, and let $\mbox{vol}(\cdot)$ denote the Lebesgue measure on $(L(\Delta), \Z(\mathcal{S}\cap\Delta))$, for which the volume of ($\dim\Delta$)-dimensional standard cube, in this measure, is equal to $1$. For a subset $P\subset L(\Delta)$, we have
	\begin{equation}\label{vol1}
		\mbox{Vol}_\mathbb{Z}(P)=(\dim\Delta)!\cdot\mbox{vol}(P).
	\end{equation}
For more details, see \cite{MTEuler,matsui2011}.
	
On the other hand, assume that $\Delta$ is $k$-dimensional. If $P\subset L(\Delta)\cong\R^k$ is a simplex with vertices $x_0,x_1,\dots,x_k\in\mathbb{R}^k$ then
$$\{x_0-x_i,x_1-x_i,\dots,x_{i-1}-x_i,x_{i+1}-x_i,\dots,x_k-x_i\}$$
is a basis of $\mathbb{R}^k$. In this case $\mbox{vol}(P)$ is given by
\begin{equation}\label{vol2}
	\mbox{vol}(P)=\dfrac{\mid\det(x_0-x_i\,\cdots\,\widehat{x_i-x_i}\,\cdots\,x_k-x_i)\mid}{k!},
\end{equation}
where $(x_0-x_i\,\cdots\,\widehat{x_i-x_i}\,\cdots\,x_k-x_i)$ is the ($k\times k$)-matrix whose columns are the vectors $x_j-x_i$ for $j\neq i$. Since the set in question forms a basis for $\R^k$, the value of $\mbox{vol}(P)$ is independent of the choice of $x_i$. For more details, see \cite[Chapter 7]{cox2005using}. From (\ref{vol1}) and (\ref{vol2}) we obtain
\begin{equation}\label{vol3}
	\mbox{Vol}_{\Z}(P)=\mid\det(x_0-x_i\,\cdots\,\widehat{x_i-x_i}\,\cdots\,x_k-x_i)\mid.
\end{equation}

\section{Toric structure of 2-generic symmetric determinantal varieties}\label{sect symm det}

In this section we describe the toric structure of 2-generic symmetric determinantal varieties. We do so by providing an explicit set of vectors whose associated toric variety coincides with the 2-generic symmetric variety. Moreover, we provide a detailed description of the cone generated by that set of vectors.

Let $S_n$ be the $\mathbb{C}$-vector space of symmetric matrices of order $n$ with entries in $\mathbb{C}$. Denote by $S_n^t$ the subset of $S_n$ consisting of symmetric matrices of rank less than $t$, where $2\leq t \leq n$. The variety $S_n^t$ is called a \emph{generic symmetric determinantal variety}.

In other words, let $L=(x_{ij})_{1\leq i,j\leq n}$, where $x_{ij}$ are indeterminates and $x_{ij}=x_{ji}$ for all $i,j$. Denote by $J_t \subset \mathbb{C}[x_{ij} : 1 \leq i,j \leq n]$ the ideal generated by the $(t \times t)$-minors of $L$. Then $S_n^t = V(J_t) \subset \mathbb{C}^N$, where $N = \dfrac{n(n+1)}{2}$.

The following are some basic properties of generic symmetric determinantal varieties.

\begin{prop}\label{propsim}
	Let $S_n^t$, with $2\leq t\leq n$, be a generic symmetrical determinantal variety. Then
	\begin{enumerate}
		\item $S_n^t$ is irreducible. Moreover, $J_t$ is a prime ideal.
		\item The dimension of $S_n^t$ is $N-\dfrac{(n-t+1)(n-t+2)}{2}$.
		\item The singular locus of $S_n^t$ is $S_n^{t-1}$.
		\item $S_n^t$ is normal.
	\end{enumerate}
\end{prop}
\begin{proof}
	These are known facts on generic symmetric determinantal varieties. See \cite[Proposition 1.1]{symmetric} for 2, 3, and the first statement of 1. For the second statement of 1, see \cite[Section 1, Corollary]{Kutz}.  The normality of $S_n^t$ is a consequence of 3 and the fact that $S_n^t$ is Cohen-Macaulay \cite[Theorem 1]{Kutz}.
\end{proof}

Because of 1 of the proposition, $S^2_n$ is a toric variety. We are interested in describing the toric structure that carries this particular case. 

\subsection{A semigroup defining the 2-generic symmetric determinantal varieties}

Our first goal is to explicitly describe a finitely generated semigroup $\mathcal{S}$ satisfying $S^2_n=\xs$. The following example illustrates the main ideas we use to prove that result.

\begin{exem}
	Consider $S^2_3=V(f_1,f_2,f_3,f_4,f_5,f_6)\subset\mathbb{C}^6$, where $f_1=z_1z_4-z_2^2$, $f_2=z_1z_5-z_2z_3$, $f_3=z_4z_6-z_5^2$, $f_4=z_2z_6-z_3z_5$, $f_5=z_1z_6-z_3^2$ and $f_6=z_2z_5-z_3z_4$. The binomials $f_1,\ldots,f_6$ correspond to the $(2\times2)$-minors of the following matrix
	$$
	\left(
	\begin{array}{ccc}
		z_1 & z_2 & z_3  \\
		z_2 & z_4 & z_5  \\
		z_3 & z_5 & z_6
	\end{array}
	\right).
	$$
	By Proposition \ref{propsim}, $\dim S^2_3=3$. Let $\mathcal{S}\subset\mathbb{Z}^3$ be the semigroup generated by $\mathscr{A}=\{m_i\}_{1\leq i\leq 6}$, where
	\begin{align*}
		m_1 & =e_1,\\
		m_2 &= e_1+e_2,\\
		m_3 &= e_1+e_3,\\
		m_4 &= e_1+2e_2,\\
		m_5 &= e_1+e_2+e_3,\\
		m_6 &= e_1+2e_3.
	\end{align*}
	Here $e_1,e_2,e_3$ denote the canonical basis of $\Z^3$. Since $\mathbb{Z}\mathscr{A}=\mathbb{Z}^3$, we have $\dim X_\mathcal{S}=3$. Furthermore, a straightforward computation shows that $f_1,\ldots,f_6\in I_\mathcal{S}$. Therefore, $X_\mathcal{S}\subset S^2_3$. As these are irreducible varieties of the same dimension, we conclude that $X_\mathcal{S}=S^2_3$.
\end{exem}

The following result was inspired by a similar treatment for 2-generic determinantal varieties, see \cite[Theorem 2.7]{DDR}.

\begin{teo}\label{esttoric}
	Let $n\geq2$ and let $\mathcal{S}\subset\mathbb{Z}^n$ be the semigroup generated by the following set
	\begin{align*}
		\mathscr{A} = \{& e_1, e_1+e_2, e_1+e_3,\dots, e_1+e_{n-1}, e_1+e_n, \\
		& e_1+2e_2, e_1+e_2+e_3, e_1+e_2+e_4,\dots, e_1+e_2+e_{n-1},e_1+e_2+e_n,\\
		& e_1+2e_3, e_1+e_3+e_4,\dots, e_1+e_3+e_{n-1}, e_1+e_3+e_n,\\
		&\qquad \vdots\\
		& e_1+2e_{n-1},e_1+e_{n-1}+e_n,\\
		& e_1+2e_n\}.                        
	\end{align*}
	Then $X_\mathcal{S}=S^2_n$.
\end{teo}
\begin{proof}
	Let $L=(x_{ij})_{1\leq i,j\leq n}$ be a symmetric matrix, where $x_{ij}$ are indeterminates. Let $J_2$ be the ideal generated by the $(2\times2)$-minors of $L$, i.e., 
	\[
	J_2=\{x_{ij}x_{kl}-x_{il}x_{kj}:1\leq i<k\leq n, 1\leq j<l\leq n\;\mbox{and}\;x_{ab}=x_{ba}\}.
	\] 
	We denote the elements of $\mathscr{A}$ as follows,
	\begin{itemize}
		\item $m_{11}=e_1$;
		\item $m_{1j}=e_1+e_j$, for $2\leq j\leq n$;
		\item $m_{ij}=e_1+e_i+e_j$, for $2\leq i\leq j\leq n$;
		\item $m_{ij}=m_{ji}$, for $2\leq i\leq j\leq n$.
	\end{itemize} 
	We claim that the multiplicative relation \(x_{ij}x_{kl} = x_{il}x_{kj}\) for \(1 \leq i < k \leq n\), \(1 \leq j < l \leq n\), corresponds to the additive relation \(m_{ij} + m_{kl} = m_{il} + m_{kj}\). To verify this, we consider the following cases:
	\begin{enumerate}
		\item $i=j=1$. By definition, $m_{kl}=m_{lk}$. Hence we can assume $k\leq l$. It follows that
		\begin{align*}
			m_{11}+m_{kl} & = e_1+(e_1+e_k+e_l)\\
			&	= (e_1+e_l)+(e_1+e_k)\\
			& = m_{1l}+m_{k1}.		
		\end{align*}
		\item $i\geq 2$ and $j=1$. As in the previous item we obtain
		\begin{align*}
			m_{i1}+m_{kl} & = (e_1+e_i)+(e_1+e_k+e_l)\\
			& = (e_1+e_i+e_l)+(e_1+e_k)\\
			& = m_{il}+m_{k1}.
		\end{align*}
		\item $i=1$ and $j\geq2$. As before,
		\begin{align*}
			m_{1j}+m_{kl} & =(e_1+e_j)+(e_1+e_k+e_l)\\
			& =(e_1+e_l)+(e_1+e_k+e_j)\\
			& = m_{1l}+m_{kj}.
		\end{align*}
		\item $i\geq2$ and $j\geq2$. As before,
		\begin{align*}
			m_{ij}+m_{kl} & = (e_1+e_i+e_j)+(e_1+e_k+e_l)\\
			& = (e_1+e_i+e_l)+(e_1+e_k+e_j)\\
			& = m_{il}+m_{kj}.
		\end{align*}
	\end{enumerate}
It follows that $J_2\subset I_{\mathcal{S}}$. Hence, $X_\mathcal{S}=V(I_\mathcal{S})\subset V(J_2)=S_n^2$. Since $\Z\A=\mathbb{Z}^n$, we have that $\dim X_\mathcal{S}=n$. We also know that $\dim S^2_n=n$ by Proposition \ref{propsim}. As \(X_\mathcal{S}\) and \(S^2_n\) are both irreducible varieties of the same dimension, we conclude that $X_\mathcal{S}=S^2_n$.  
\end{proof}

\begin{cor}\label{cornorm}
$\xs$ is a normal variety. In particular, $\mathcal{S}=\cone(\A)\cap\Z^{n}$. Moreover, $\cone(\A)$ is a strongly convex cone.
\end{cor}
\begin{proof}
	The first part follows from Theorem \ref{esttoric} and Propositions \ref{proptor} and \ref{propsim}. The second part follows from the fact $0\in S^2_n=\xs$.
\end{proof}

\begin{cor}
	$\A$ is the minimal generating set of $\mathcal{S}$.
\end{cor}
\begin{proof}
Let $\sigma\subset\R^{n}$ be the dual cone of $\cone(\A)$. Since $\cone(\A)$ is a full-dimensional and strongly convex cone, $\sigma$ also satisfy these properties. The previous corollary shows that $\xs$ is the normal toric variety defined by $\sigma$. 
	
Let $\mathcal{H}$ be the minimal generating set of $\mathcal{S}$. Let $T_0\xs\subset\C^{N}$ be the tangent space of $\xs$ at the origin. It is known that $\dim_{\C}T_0\xs=|\mathcal{H}|$ \cite[Lemma 1.3.10]{cox2011}. We claim that $T_0\xs=\C^{N}$. This implies $|\mathcal{H}|=N$. Since $\mathcal{H}\subset\A$ and $|\A|=N$ we conclude $\mathcal{H}=\A$.
	
To prove the claim recall that $J_2$ and $I_{\mathcal{S}}$ are prime ideals. Hence, $J_2=\mathbb{I}(S^2_n)=\mathbb{I}(\xs)=I_{\mathcal{S}}$. Since the generators of $J_2$ are of the form $x_{ij}x_{kl}-x_{il}x_{kj}$, we have that the Jacobian matrix of $J_2$, evaluated at $0$, is the zero matrix. In particular, its kernel is $\C^{N}$, proving the claim.
\end{proof}

\subsection{Some properties of the cone generated by $\mathcal{S}$}

We keep the notation from Theorem \ref{esttoric}. Denote $K(\mathcal{S})= \operatorname{Cone}(\A)$. As mentioned earlier, $K(\mathcal{S})$ is a strongly convex cone in $\mathbb{R}^n$ and  $\dim K(\mathcal{S}) = n$. In the following results we study the cone structure of $\KS$.

\begin{lem}\label{lemrays}
	Let $\mathscr{A}$ and $\mathcal{S}$ be as in Theorem \ref{esttoric}. The vectors in $\A$ of the form $e_1$ and $e_1 + 2e_j$, for $2 \leq j \leq n$, generate the rays of $K(\mathcal{S})$. In particular, $\KS$ is a simplicial cone.
\end{lem}
\begin{proof}
	Firstly, the vectors in $\A$ of the form $e_1 + e_i + e_j$,  for $2\leq i<j\leq n$, do not generate rays. Indeed,
	$$
	e_1+e_i+e_j=\dfrac{1}{2}(e_1+2e_i)+\dfrac{1}{2}(e_1+2e_j).
	$$
	Similarly, $e_1+e_i$, for $2\leq i\leq n$, are not rays of $\KS$ since
	$$
	e_1+e_i=\dfrac{1}{2}e_1 +\dfrac{1}{2}(e_1+2e_i).
	$$
	Now let $x=(x_1,\dots,x_n)\in\mathbb{R}^n$ and consider the linear maps $L_1: \mathbb{R}^n\longrightarrow\mathbb{R}$ and $L_j: \mathbb{R}^n\longrightarrow\mathbb{R}$, $2 \leq j \leq n$, defined as
	$$
	L_1(x)=\displaystyle\sum_{2\leq i\leq n}\!\!\!x_i 
	\;\;\;\mbox{and}\;\;\;
	L_j(x) =2x_1-x_j+\displaystyle\sum_{\begin{array}{c}
			\scriptstyle 2\leq i\leq n \\[0cm]
			\scriptstyle i\neq j\;\; 
	\end{array}}\!\!\!\!\! x_i,\;\mbox{for}\; 2\leq j\leq n.
	$$
	The following properties hold,
	\begin{itemize}
		\item $L_1(e_1)=0$,
		\item $L_1(e_1+e_j)>0$ for $2\leq j\leq n$,
		\item $L_1(e_1+e_i+e_j)>0$ for $2\leq i\leq j\leq n$,
		\item $L_j(e_1+2e_j)=0$,
		\item $L_j(e_1)>0$,
		\item $L_j(e_1+e_i)>0$ for $2\leq i \leq n$,
		\item $L_j(e_1+e_i+e_j)>0$ for $2\leq i<j\leq n$.        
	\end{itemize}
	By the previous items, $H_1=\{L_1=0\}$ and $H_j=\{L_j=0\}$ are supporting hyperplanes of $\KS$, and $\KS\cap H_1=\R_{\geq0}e_1$ and $\KS\cap H_j=\R_{\geq0}(e_1+2e_j)$. This proves the lemma.
	%Thus, the vectors of the form $e_1$ and $e_1 + 2e_j$ for $2 \leq j \leq n$ generate the rays of $K(\mathcal{S})$.
\end{proof}

\begin{cor}\label{corrays}
	Let $\mathscr{A}$ and $\mathcal{S}$ be as in Theorem \ref{esttoric}, and let $0\leq k \leq n$. The cone $K(\mathcal{S})$ has exactly $\binom{n}{k}$ $k$-dimensional faces.
\end{cor}
\begin{proof}
	The only face of $\KS$ of dimension $n$ is $\KS$ itself. Since $\KS$ is strongly convex, $\{0\}$ is a 0-dimensional face. By the previous lemma, there are $n$ 1-dimensional faces. Now assume $1<k<n$. We first construct $\binom{n}{k}$ $k$-dimensional faces of $\KS$. By Lemma \ref{lemrays}, the generators of the rays of $\KS$ are
	$$R(\KS)=\{e_1,e_1+2e_2,e_1+2e_3,\ldots,e_1+2e_n\}.$$
	
	Let $\tau$ be the cone generated by the following vectors, where $2\leq i_2<\cdots<i_k \leq n$,
	\begin{eqnarray}
		e_1, e_1+2e_{i_2},\dots,e_1+2e_{i_k}. \label{tp1}
	\end{eqnarray}
	These vectors are linearly independent, hence $\dim \tau=k$. We now show that $\tau$ is a face of $\KS$. Consider the linear map $L_{\tau}:\mathbb{R}^n\longrightarrow\mathbb{R}$ given by $L_{\tau}(x)=\sum_{i\not\in\{1,i_2,\dots,i_k\}}x_i$. A straightforward computation shows that $L_{\tau}(x)=0$ for $x\in\tau$ and $L_{\tau}(x)>0$ for $x\in K(\mathcal{S})\setminus\tau$. Thus, $\tau$ is a face of $\KS$. Notice that there are $\binom{n-1}{k-1}$ such possible faces, corresponding to the possible choices as in (\ref{tp1}).
	
	Now let $\tau'$ be the cone generated by the following vectors, where $2\leq j_1<\cdots<j_k \leq n$,
	\begin{eqnarray}
		e_1+2e_{j_1},\dots,e_1+2e_{j_k}. \label{tp2}
	\end{eqnarray}
	We proceed as in the previous paragraph to show that $\tau'$ is a $k$-dimensional face of $\KS$. These vectors are linearly independent, hence $\dim \tau'=k$. Consider the linear map $L_{\tau'}:\mathbb{R}^n\longrightarrow\mathbb{R}$, $L_{\tau'}(x)=2x_1-\sum_{i\in\{j_1,\dots,j_{k}\}}x_i$. As before, $L_{\tau'}(x)=0$ for $x\in\tau'$ and $L_{\tau'}(x)>0$ for $x\in K(\mathcal{S})\setminus\tau'$. Notice that there are $\binom{n-1}{k}$ such possible faces, corresponding to the possible choices as in (\ref{tp2}).
	
	Finally, counting all possible combinations (\ref{tp1}) and (\ref{tp2}) results in a total of
	$$
	\binom{n-1}{k-1}+\binom{n-1}{k}=\binom{n}{k}
	$$ 
	$k$-dimensional faces. 
	
	Notice that there are no other $k$-dimensional faces in $\KS$. Indeed, the previous paragraphs exhibited all $k$-dimensional faces generated by exactly $k$ rays. Let $\tau$ be a face of $\KS$ generated by more than $k$ rays. Since $R(\KS)$ is linearly independent, it follows that $\dim \tau>k$.
\end{proof}

\begin{prop}\label{propface}
	Let $2\leq k\leq n$ and let $\Delta_k$ be a $k$-dimensional face of $\KS$. Then $X_{\Delta_k}=S^2_k$.
\end{prop} 
\begin{proof}
	By Corollary \ref{corrays}, $\Delta_k$ is generated by subsets of $R(\KS)$ as in (\ref{tp1}) or (\ref{tp2}). First consider subsets of the form (\ref{tp1}). Without loss of generality, we may assume that
	$$
	\Delta^1_k=\R_{\geq0}(e_1,e_1+2e_2,\dots,e_1+2e_k).
	$$
	Notice that $\Delta^1_k\subset\R^{k}\times\{(0,\ldots,0)\}$, where $(0,\ldots,0)\in\R^{n-k}$. Hence we assume $\dku\subset\R^k$. We claim that the semigroup $\dku\cap\Z^k$ has as generators the set $\A$ of Theorem \ref{esttoric} for $n=k$. By Lemma \ref{lemrays}, we know that $\KS=\dku$ in the case $n=k$. Then Corollary \ref{cornorm} implies the claim. We conclude that $X_{\dku}=S^2_k$.
	
	We now consider a $k$-dimensional face of $K(\mathcal{S})$ generated by vectors of the form (\ref{tp2}). Let
	$$
	\Delta^2_k=\R_{\geq0}(e_1+2e_2,\dots, e_1+2e_{k+1}).
	$$
	As before, we assume that $\dkd\subset\Z^{k+1}$. Let 
	$$\mathscr{A}^2_k=\{e_1+2e_2,\dots, e_1+2e_{k+1}\}\cup\{e_1+e_i+e_j\mid 2\leq i<j\leq k+1\}\subset\Z^{k+1}.$$
	As in the proof of Lemma \ref{lemrays}, we have that $\mathscr{A}^2_k\subset\dkd\cap\Z^{k+1}$. Moreover, $\dkd=\R_{\geq0}\mathscr{A}^2_k$. We claim that $X_{\N\mathscr{A}^2_k}=S^2_k$. Assuming the claim for the moment, we obtain that $\N\mathscr{A}^2_k$ is a saturated semigroup since $S^2_k$ is normal. Thus, $\Z^{k+1}\cap\dkd=\Z^{k+1}\cap\R_{\geq0}\mathscr{A}^2_k=\N\mathscr{A}^2_k$ and so $X_{\dkd}=X_{\N\mathscr{A}^2_k}=S^2_k$. Now we prove the claim.
	
	First notice that $|\mathscr{A}^2_k|=\binom{k}{2}+k=\frac{k(k+1)}{2}.$ In particular, both $S^2_k$ and $X_{\N\mathscr{A}^2_k}$ are embedded in $\C^{\frac{k(k+1)}{2}}$. 	Let $J_2 \subset \mathbb{C}[x_{ij} : 1 \leq i,j \leq k]$ be the ideal generated by the $(2 \times 2)$-minors of a $k\times k$ symmetric matrix $L$ whose entries are the variables $x_{ij}$.
	
	Denote the elements of $\mathscr{A}^2_k$ as follows, $m_{ij}=e_1+e_i+e_j$, for $2\leq i\leq j\leq k$. Notice that $m_{ji}=m_{ij}$. Analogous computations to those of the proof of Theorem \ref{esttoric} show that $J_2\subset I_{\N\mathscr{A}^2_k}$. Therefore, $X_{\N\mathscr{A}^2_k}\subset S^2_k$. Since these are irreducible varieties, to obtain that they are equal we show that $\operatorname{rank}(\Z\mathscr{A}^2_k)=k$.
	
	Define the set $\beta=\{e_1+e_2+e_j:2\leq j\leq k+1\}$.  We show that $\Z\beta=\Z\mathscr{A}^2_k$. This follows from
	$$e_1+2e_j=2(e_1+e_2+e_j)-(e_1+2e_2),\,\,\, \mbox{for }\,\,\, j>2, $$
	and
	$$e_1+e_i+e_j =(e_1+e_2+e_i)+(e_1+e_2+e_j)-(e_1+2e_2).$$
	Since $\beta$ is a linearly independent subset of $\mathbb{Z}^{k+1}$ and $|\beta|=k$ we conclude that $\operatorname{rank}(\Z\mathscr{A}^2_k)=k$.
\end{proof}

%%%%%%%%%%%%%%%%%%%%%%%%%%%%%%%%%
%%%%%%%%%%%%%%%%%%%%%%%%%%%%%%%%%

\section{The Euler characteristic of a Milnor fiber on $S_n^ 2$}

In this section we use the toric structure of $S_n^2$ to establish a formula for the Euler characteristic of the Milnor fiber of non-degenerate polynomial functions having an isolated critical point at the origin relative to a stratification. Our first goal is to describe explicitly such polynomials.

Let $X_\mathcal{S}\subset\mathbb{C}^N$ be a toric variety, where $\sgs=\N(\{m_1,\ldots,m_N\})\subset\Z^n$. Let $\KS=\cone(m_1,\ldots,m_N)$. Assume that $\ks$ is strongly convex. A Whitney stratification of $\xs$ is given by 
\begin{align}\label{Wh strat}
\mathcal{V}_\mathcal{S} = \{T_{\lambda}\}_{\lambda \prec \KS},
\end{align}
where ${\lambda \prec \KS}$ denotes a face of $\KS$ and $T_{\lambda}$ is the orbit of the action of the torus associated to $\lambda$.

Let $f:\xs\rightarrow\mathbb{C}$ be a polynomial function. The critical locus of $f$ relative to $\mathcal{V_{\ms}}$ is defined as the union 
	\begin{equation*}
		\Sigma_{\mathcal{V_{\ms}}}f=\bigcup_{T_{\lambda}\in\mathcal{V}_{\ms}}\Sigma(f|_{T_{\lambda}}),
	\end{equation*}
where $\Sigma(f|_{T_{\lambda}})$ is the set of critical points of $f|_{T_{\lambda}}$ (see \cite{MR1404920}, where this relative notion of critical points is introduced in a more general context).

Let us apply the previous concepts to the toric varieties $S^2_n$. Let $\mathscr{A}$ and $\mathcal{S}$ be as in Theorem \ref{esttoric}. By Lemma \ref{lemrays},
the vectors in $\A$ of the form $e_1$ and $e_1 + 2e_j$, for $2 \leq j \leq n$, generate the $n$ rays of $K(\mathcal{S})$. From now on we order the elements of $\mathscr{A}$ in the following way
$$\A = \{m_1, m_2, \dots, m_n, m_{n+1}, \dots, m_{N}\},$$ 
where $m_1=e_1$ and $m_j=e_1+2e_j$, for $2 \leq j \leq n$. 

\begin{lem}\label{lemsi}
Let $f:S^2_n \to \mathbb{C}$ be a polynomial function. If $\Sigma_{\mathcal{V_{\ms}}}f=\{0\}$, then $f$ must contain a pure monomial in the variable $z_i$, for all $1 \leq i \leq n$. In other words, $f$ contains terms of the form $ \alpha_i z_i^{d_i},$ where $\alpha_i \neq 0$ and $d_i \geq 1$, for all $1 \leq i \leq n$.
\end{lem}
\begin{proof}
Consider the following sets:
$$T_i = \{ (0, \dots, 0, z_i, 0, \dots, 0)\in\C^N\mid \ \ z_i \neq 0, \ \ 1 \leq i \leq n\}.$$
Knowing that $m_1,\ldots,m_n$ determines the rays of $\KS$, from the defining equations of $S_n^2$ we deduce that $T_i\subset S_n^2.$

Recall the action of the torus $\theta: (\mathbb{C}^*)^n \times S_n^2 \to S_n^2 $, $\theta(t, z) = (t^{m_1} z_1, \dots, t^{m_N} z_N).$ It follows that the $T_i's$ are in fact the $1$-dimensional orbits of the action on $S_n^2$. In particular, the $T_i$'s are strata of the Whitney stratification $\mathcal{V}_{\mathcal{S}}$. Therefore, if $f$ does not contain a pure monomial in the variable $z_i$ for some $1 \leq i \leq n$, then the $1$-dimensional stratum $T_i$ is contained in $\Sigma_{\mathcal{V}_{\mathcal{S}}}f$, which is a contradiction.
\end{proof}

\begin{prop}\label{form polyn}
Let $f: S^2_n \to \mathbb{C}$ be a non-degenerate polynomial function such that $\Sigma_{\mathcal{V_{\ms}}}f=\{0\}$. Then $f$ has the form 
	\[
	f(z)=\sum_{i=1}^{n}\alpha_i z_i^{d_i}+h(z),
	\]
where $\alpha_i \neq 0$ and $d_i\geq1$ for all $1\leq i \leq n$, and $h(z)$ is a polynomial function satisfying $\Gamma_+(h) \subset \Gamma_+(\sum_{i=1}^{n}z_i^{d_i})$.    
\end{prop}
\begin{proof}
If the critical locus of $f$ relative to $\mathcal{V}_S$ is $\{0\}$, then by Lemma \ref{lemsi}, $f$ contains terms of the form $ \alpha_i z_i^{d_i},$ where $\alpha_i \neq 0$ and $d_i \geq 1$, for all $1 \leq i \leq n$. Notice that, for each $i$, there might be several terms of that form. We can write $f$ as follows, where each term $\alpha_i z_i^{d_i}$ corresponds to the minimal value of $d_i$, for each $i\in\{1, \dots, n\}$:
	\[
	f(z)=\sum_{i=1}^{n}\alpha_i z_i^{d_i}+h(z).
	\]
In the previous expression we have: $\alpha_i \neq 0$, $d_i\geq1$ for all $1\leq i \leq n$, and $h(z)$ is a polynomial function satisfying $\Gamma_+(h) \subseteq \Gamma_+(\sum_{i=1}^{n}{\alpha_i}z_i^{d_i})$. In other words, $\Gamma_+(f) = \Gamma_+\left(\sum_{i=1}^{n}{\alpha_i} z_i^{d_i}\right)$. Indeed, if $\Gamma_+(f) \cap \mathbb{R}_{+}(m_i) = \emptyset$, for some $i\in\{1, \dots, n\}$, it would imply that the associated $1$-dimensional orbit $T_i = \left\{ (0, \dots, 0, z_i, 0, \dots, 0) \in \mathbb{C}^N \mid z_i \neq 0,\ 1 \leq i \leq n \right\}$ is contained in the critical locus of $f$, which is a contradiction.	
\end{proof}

%\begin{obs}
%Actually, any non-degenerate polynomial function $f:S^2_n\to \C$ having the form described in Proposition \ref{form polyn} has $\{0\}$ as critical locus relative to $\mathcal{V}_S$. Indeed, in this case, $\Gamma_+(f) \cap \lambda \neq \emptyset$ for every face $\lambda \prec K(\mathcal{S})$ of the cone $K(\mathcal{S})$, with $\lambda \neq \{0\}$. Such a non-empty intersection implies that the restriction of $f$ to the orbit $T_\lambda$, associated with the torus action of $(\mathbb{C}^*)^n$ on $X_S$, is not identically zero, i.e., $f|_{T_\lambda} \not\equiv 0$. Furthermore, since $f$ is non-degenerate, it follows from \cite[Proposition 15]{DalbelodaNobrega} that there exists a positive number $R > 0$ such that $f$ has no singularities in $T_\lambda \cap B_R$, where $B_R$ denotes the open ball of radius $R$ centered at the origin $0 \in \mathbb{C}^N$. Therefore, $f$ has no critical points along any orbit $T_\lambda$ other than the origin.	
%\end{obs}

\subsection{A formula for the Euler characteristic of the Milnor fiber on $S^2_n$}

We are now ready to provide a formula for the Euler characteristic of the Milnor fiber of a non-degenerate polynomial function on $S^2_n$ having an isolated critical point at the origin relative to $\mathcal{V}_{\mathcal{S}}$.

\begin{teo}\label{differentds}
Let $S^2_n\subset\mathbb{C}^N$ be a 2-generic symmetric determinantal variety. Consider a non-degenerate polynomial function $f:S^2_n\longrightarrow\mathbb{C}$ given by 
	\[
	f(z)=\sum_{i=1}^{n}\alpha_i z_i^{d_i}+h(z),
	\]
where $\alpha_i \neq 0$ and $d_i\geq1$ for all $1\leq i \leq n$, and $h(z)$ is a polynomial function satisfying $\Gamma_+(h) \subset \Gamma_+(\sum_{i=1}^{n}z_i^{d_i})$. Then the Euler characteristic of the Milnor fiber $F_0$ of $f$ in $0\in S_n^2$ is given by
	\[
	\chi(F_0)=\sum_{k=1}^n(-1)^{k-1}2^{k-1}\sum_{1\leq i_1<\dots <i_k\leq n}^{}d_{i_1} \dots d_{i_k}.
	\]
In particular, this formula holds for non-degenerate polynomial functions on $S^2_n$ with isolated critical point at the origin relative to $\mathcal{V}_{\ms}$.
\end{teo}
\begin{proof}
The last statement of the Theorem follows from Proposition \ref{form polyn}.

Let $\A$ and $\sgs$ be as in Theorem \ref{esttoric}, so that $S^2_n=X_{\sgs}$. Let $\Gamma_+(f)$ be the Newton polyhedron of $f$. Since the polynomial $f$ is non-degenerate, by Theorem \ref{matsui2011}, we have
	\begin{equation*}\label{formbruta}
		\chi(F_0)= \sum_{\Gamma_+(f)\cap\Delta\neq\emptyset}(-1)^{\dim\Delta-1}\sum_{i=1}^{v(\Delta)}\mbox{Vol}_\mathbb{Z}(\Gamma^\Delta_i).
	\end{equation*}
Moreover, since $\Gamma_+(f)$ intersects all rays of $K(\mathcal{S})$, it follows that $\Gamma_+(f)\cap\Delta\neq\emptyset$ for any non-zero face $\Delta$ of $K(\mathcal{S})$.

On the other hand, for each $k$-dimensional face $\Delta$ of $\KS$, $1\leq k\leq n$, there exists a unique compact face of $\Gamma_+(f)\cap\Delta$ of dimension $k-1$. Indeed, by Lemma \ref{lemrays}, $\Delta$ is simplicial. Let $\{g_{i_1},\ldots,g_{i_k}\}$ be the ray generators of $\Delta$. Then, the only compact face of $\Gamma_+(f)\cap\Delta$ of dimension $k-1$ is the convex hull of $\{d_{i_1}g_{i_1},\ldots,d_{i_k}g_{i_k}\}$. 
	
	For each $1\leq k\leq n$, let $\{\Delta^{(k)}_i|1\leq i\leq \binom{n}{k}\}$ be the set of $k$-dimensional faces of $\KS$ (see Corollary \ref{corrays}). For each $\Delta^{(k)}_i$, let $\gamma^{\Delta^{(k)}_i}$ be the only compact face of $\Gamma_+(f)\cap\Delta^{(k)}_i$ of dimension $k-1$. Let $\Gamma^{\Delta^{(k)}_i}$ be the convex hull of $\gamma^{\Delta^{(k)}_i}\cup\{0\}$. Then, the previous formula can be rewritten as
	\begin{equation}\label{form1lapida}
		\chi(F_0)= \sum_{k=1}^{n}(-1)^{k-1}\sum_{i=1}^{\binom{n}{k}}\mbox{Vol}_\mathbb{Z}(\Gamma^{\Delta^{(k)}_i}).
	\end{equation}
Now we compute $\mbox{Vol}_\mathbb{Z}(\Gamma^{\Delta^{(k)}_i})$ (recall the discussion at the end of Section \ref{MT formula}).

If $k=1$, $\Gamma^{\Delta^{(1)}_i}$ are segments joining the origin to  $d_1e_1$ or $d_i(e_1+2e_i)$, $2\leq i\leq n$, in $\R^n$. Then
\[\mbox{Vol}_\Z(\Gamma^{\Delta^{(1)}_i})=d_i.\]
Let $2\leq k\leq n$. By the proof of Proposition \ref{propface}, there are two types of generators for a $k$-dimensional face of $\KS$, namely
\begin{equation}\label{tipo1}
	\{e_1,e_1+2e_{i_2},\dots,e_1+2e_{i_k}\}, \mbox{ for } 2\leq i_2<\dots<i_k\leq n
\end{equation}
and
\begin{equation}\label{tipo2}
	\{e_1+2e_{j_2},\dots,e_1+2e_{j_k}\}, \mbox{ for } 2\leq j_2<\dots<j_k\leq n
	\end{equation}
	Let $\Delta^{(k)}_\tau$ and $\Delta^{(k)}_{\tau'}$ be the faces generated by vectors of type (\ref{tipo1}) and (\ref{tipo2}), respectively. For the face $\Delta^{(k)}_\tau$ we have
	$$
	\mbox{Vol}_\Z(\Gamma^{\Delta^{(k)}_\tau})=\mbox{Vol}_\Z\left\{0,
	\begin{pmatrix}
		d_1\\
		0\\
		\vdots\\
		0\\
		0\\
		0\\
		\vdots\\
		0\\
		0\\
		0\\
		\vdots\\
		0\\
		0
	\end{pmatrix},
	\begin{pmatrix}
		d_{i_2}\\
		0\\
		\vdots\\
		0\\
		2d_{i_2}\\
		0\\
		\vdots\\
		0\\
		0\\
		0\\
		\vdots\\
		0\\
		0
	\end{pmatrix},\dots,
	\begin{pmatrix}
		d_{i_k}\\
		0\\
		\vdots\\
		0\\
		0\\
		0\\
		\vdots\\
		0\\
		2d_{i_k}\\
		0\\
		\vdots\\
		0\\
		0
	\end{pmatrix}\right\}.
	$$
On a appropriate lattice, using (\ref{vol3}) we obtain $\mbox{Vol}_\Z(\Gamma^{\Delta^{(k)}_\tau})=2^{k-1}d_1d_{i_2}\dots d_{i_k}$. On the other hand, for the face $\Delta^{(k)}_{\tau'}$ we have to compute the volume

	$$
	\mbox{Vol}_\Z(\Gamma^{\Delta^{(k)}_{\tau'}})=\mbox{Vol}_\Z\left\{0,
	\begin{pmatrix}
		d_{j_2}\\
		0\\
		\vdots\\
		0\\
		2d_{j_2}\\
		0\\
		\vdots\\
		0\\
		0\\
		0\\
		\vdots\\
		0\\
		0\\
		0\\
		\vdots\\
		0\\
		0
	\end{pmatrix},
	\begin{pmatrix}
		d_{j_3}\\
		0\\
		\vdots\\
		0\\
		0\\
		0\\
		\vdots\\
		0\\
		2d_{j_3}\\
		0\\
		\vdots\\
		0\\
		0\\
		0\\
		\vdots\\
		0\\
		0
	\end{pmatrix},\dots,
	\begin{pmatrix}
		d_{j_{k+1}}\\
		0\\
		\vdots\\
		0\\
		0\\
		0\\
		\vdots\\
		0\\
		0\\
		0\\
		\vdots\\
		0\\
		2d_{i_{k+1}}\\
		0\\
		\vdots\\
		0\\
		0
	\end{pmatrix}\right\}.
	$$
In contrast to the previous case, the lattice generated by $\Delta^{(k)}_{\tau'}$ is not contained in any coordinate subspace. However, in the proof of Proposition \ref{propface} we showed that this lattice is isomorphic to $\Z^k$ with base 
	\[
	\beta=\{e_1+2e_{j_2}, e_1+e_{j_2}+e_{j_3},\dots,e_1+e_{j_2}+e_{j_k}\}.
	\]
	Thus, expressing the generators of $\Delta^{(k)}_{\tau'}$ in terms of basis $\beta$, we obtain:
	\begin{itemize}
		\item for $d_{j_2}e_1+2d_{j_2}e_{j_2}$
		$$
		d_{j_2}e_1+2d_{j_2}e_{j_2} =d_{j_2}(e_1+2e_{j_2});
		$$
		\item for $d_{j_3}e_1+2d_{j_3}e_{j_3}$
		\begin{eqnarray*}
			d_{j_3}e_1+2d_{j_3}e_{j_3} & = & d_{j_3}(e_1+2e_{j_3})\\
			& = &d_{j_3}[-(e_1+2e_{j_2})+2(e_1+e_{j_2}+e_{j_3})]\\
			& =& -d_{j_3}(e_1+2e_{j_2}) +2d_{j_3}(e_1+e_{j_2}+e_{j_3});
		\end{eqnarray*}
		$\vdots$
		\item for $d_{j_{k+1}}e_1+2d_{j_{k+1}}e_{j_{k+1}}$
		\begin{eqnarray*}
			d_{j_{k+1}}e_1+2d_{j_{k+1}}e_{j_{k+1}} & = & d_{j_{k+1}}(e_1+2e_{j_{k+1}})\\
			& = & d_{j_{k+1}}[-(e_1+2e_{j_2})+2(e_1+e_{j_2}+e_{j_{k+1}})]\\
			& =& -d_{j_{k+1}}(e_1+2e_{j_2}) +2d_{j_{k+1}}(e_1+e_{j_2}+e_{j_{k+1}}).
		\end{eqnarray*}
	\end{itemize}
	Thus, using (\ref{vol3})
	\begin{align*}
		\mbox{Vol}_\Z(\Gamma^{\Delta^{(k)}_{\tau'}}) &=\mbox{Vol}_\Z\left\{0,
		\begin{pmatrix}
			d_{j_2}\\
			0\\
			0\\
			\vdots\\
			0\\
			0\\
			0
		\end{pmatrix},
		\begin{pmatrix}
			-d_{j_3}\\
			2d_{j_3}\\
			0\\
			\vdots\\
			0\\
			0\\
			0
		\end{pmatrix},\dots,
		\begin{pmatrix}
			d_{j_{k+1}}\\
			0\\
			0\\
			\vdots \\
			0\\
			0\\
			2d_{j_{k+1}}
		\end{pmatrix}\right\}\\
		&= 2^{k-1}d_{j_2}\dots d_{j_{k+1}}.    
	\end{align*}
	Therefore, substituting these volumes in (\ref{form1lapida}), we conclude
	$$
	\chi(F_0)=\sum_{k=1}^n(-1)^{k-1}2^{k-1}\sum_{1\leq i_1<\dots <i_k\leq n}^{}d_{i_1} \dots d_{i_k}.
	$$	
\end{proof}

%As it will be used later on, we recall the isomorphism of the coordinate ring of $\xs$ and the semigroup algebra $\C[x^{\sgs}]$:
%\[
%\begin{array}{cccc}
%	&    \mathbb{C}[z_1,\dots,z_N]/I_{\sgs}& \longrightarrow & \C[x^{\sgs}]\\
%& z_i&   \longmapsto   & x^{m_i}.	
%\end{array}
%\]

\section{Local Euler obstruction and Milnor number} 

In this final section we deduce some consequences of Theorem \ref{differentds}. We compute the local Euler obstruction of $S^2_n$ and the Euler obstruction of the function $f$ appearing in Theorem \ref{differentds}. Moreover, we also relate the local Euler obstruction of $f$ to the Milnor number of a function defined on $\C^n$ associated to $f$.

Recall that the local Euler obstruction at a point $p \in X$ measures the obstruction to extending a radial vector field defined around $p$ to a non-vanishing vector field on $X$. On the other hand, the local Euler obstruction of a function $f$ with an isolated critical point, is a concept that extends the local Euler obstruction of $X$ by considering a different  vector field associated to $f$. For an overview on Euler obstruction, see \cite{B}.

\subsection{Local Euler obstruction}

We want to study the Euler characteristic of the Milnor fiber of generic linear forms on $S^2_n$. More generally, we have the following result.

\begin{prop}
Let $S^2_n\subset\mathbb{C}^N$ be a 2-generic symmetric determinantal variety. Let $f:S^2_n\longrightarrow\mathbb{C}$ be a non-degenerate polynomial function given by
	\[
	f(z)=\sum_{i=1}^{n}z_i+h(z),
	\]
where $h(z)$ is a polynomial function satisfying $\Gamma_+(h) \subset \Gamma_+(\sum_{i=1}^{n}z_i)$. Then the Euler characteristic of the Milnor fiber $F_0$ of $f$ at $0\in S_n^2$ is equal $1$ if $n$ is an odd number and $0$ if $n$ is an even number.
\end{prop}
\begin{proof}
By Theorem \ref{differentds}, 
%and a direct application of the binomial theorem 
we have
$$
\chi(F_0)=\binom{n}{1}-\binom{n}{2}2+\dots+(-1)^{n-2}\binom{n}{n-1}2^{n-2}+(-1)^{n-1}\binom{n}{n}2^{n-1}.
$$
On the other hand, by the Binomial Theorem,
$$
(1-2)^n= 1-\binom{n}{1}2+\binom{n}{2}2^2+\cdots+(-1)^{n-1}\binom{n}{n-1}2^{n-1}+(-1)^n2^n.
$$
Thus,
$$
\frac{1}{2}(1-(-1)^n)=\binom{n}{1}-\binom{n}{2}2+\dots+(-1)^{n-2}\binom{n}{n-1}2^{n-2}+(-1)^{n-1}\binom{n}{n}2^{n-1}.
$$
%\begin{align*}
 %   \chi(F_0)&=\binom{n}{1}-\binom{n}{2}2+\dots+(-1)^{n-2}\binom{n}{n-1}2^{n-2}+(-1)^{n-1}\binom{n}{n}2^{n-1}\\
  %  &=\frac{1}{2}((2-1)^n-(-1)^n).
%\end{align*}

The result follows from this formula.
\end{proof}

Recall that, by Proposition \ref{propsim}, the singular locus of $S_n^2\subset\C^N$ is the generic symmetric determinantal variety $S_n^1$, which is just the origin. Then, a Whitney stratification of $S_n^2$ is the set $\{S_n^2\setminus\{0\}, \{0\}\}$. Moreover, given a generic linear form $l : (\mathbb{C}^N, 0) \to (\mathbb{C}, 0)$ with respect to $S_n^2$, we can assume that
	$$l(z) = \sum_{i=1}^{N} \alpha_i z_i$$
where $\alpha_i \neq 0$, for all $i=1,\dots,N$. In addition, without loss of generality, we can assume that the restriction of $l$ to $S_n^2$ is non-degenerate. Therefore the next result follows from \cite[Theorem 3.1]{BLS} and the previous proposition. 
	
\begin{cor}\label{corobs}
Let $S^2_n\subset\mathbb{C}^N$ be a 2-generic symmetric determinantal variety. Then, ${\rm Eu}_{S^2_n}(0)$ equals $1$ if $n$ is an odd number and $0$ if $n$ is an even number.
\end{cor}

\begin{obs}
More generally, the local Euler obstruction of generic symmetric determinantal varieties was computed in \cite{Raicu}. Their approach is to study the structure of the invariant de Rham complex and character formulas for simple equivariant $\mathcal{D}$-modules.
\end{obs}

As a direct consequence of Theorem \ref{differentds}, Corollary \ref{corobs} and \cite[Theorem 3.1]{BMPS}, we obtain

\begin{cor}\label{coreulerf}
Let $f:S^2_n\longrightarrow\mathbb{C}$ be as in the Theorem \ref{differentds}. Then the local Euler obstruction of $f$ is
\[
Eu_{f,S_n^2}(0)= 1 -\sum_{k=1}^n(-1)^{k-1}2^{k-1}\sum_{1\leq i_1<\dots <i_k\leq n}^{}d_{i_1} \dots d_{i_k}
\]
if $n$ is an odd number, and
\[
Eu_{f,S_n^2}(0)= - \sum_{k=1}^n(-1)^{k-1}2^{k-1}\sum_{1\leq i_1<\dots <i_k\leq n}^{}d_{i_1} \dots d_{i_k}
\]
if $n$ is an even number.
\end{cor}

\subsection{Relation to Milnor number}

As a final application of the formulas we established earlier, we relate the local Euler obstruction of a function with isolated critical points defined on $S^2_n\subset\C^N$ to the Milnor number of a function defined on $\C^n$.

Let us recall that if $h:(\mathbb{C}^n, 0) \to (\mathbb{C}, 0)$ is an analytic function with isolated critical point at the origin, then $$ Eu_{h,{\mathbb{C}^n}}(0) = (-1)^n \mu(h),$$ where $\mu(h)$ denotes the Milnor number of $h$ \cite[Remark 3.4]{BMPS}. However, this equality is not necessarily true for arbitrary germs of complex analytic spaces.

In the following result we establish a relation among the local Euler obstruction of the function $f:S^2_n\to\C$ of Theorem \ref{differentds}, and the Milnor number of the function $g:\mathbb{C}^n \to \mathbb{C}$ given by 
	\[
	g(x)= x_1^{d_1} + \sum_{i=2}^{n}x_i^{2d_i},\,\,\, \mbox{ where }\,\,\, d_1\geq2.
	\]

\begin{prop}
With the previous notation,
\[
Eu_{f,S_n^2}(0)= - \mu(g) + \sum_{i=2}^{n} d_i +  \sum_{k=2}^{n}(-1)^{k-1}2^{k-1}\sum_{2\leq j_2<\dots <j_k\leq n}^{}d_{j_2} \dots d_{j_k},
\]
if $n$ is odd, and 
\[
Eu_{f,S_n^2}(0)= - 1 + \mu(g) + \sum_{i=2}^{n} d_i +  \sum_{k=2}^{n}(-1)^{k-1}2^{k-1}\sum_{2\leq j_2<\dots <j_k\leq n}^{}d_{j_2} \dots d_{j_k},\]
if $n$ is even. 
\end{prop}
\begin{proof} 
To begin with, we compute the Euler characteristic of the Milnor fiber of $g$ at $0\in\C^n$. Using that $\C^n$ is a toric variety defined by the semigroup $\N^n$, with associated cone $\R^n_{\geq0}$, the computation of the Euler characteristic is done exactly as in the proof of Theorem \ref{differentds}. In order not to repeat the proof, we only summarize the main steps in this context. 

First notice that $\Delta\cap \Gamma_+(g) \neq \emptyset$ for all positive-dimensional faces $\Delta$ of the cone $\R^n_{\geq0}$. Hence
	\begin{equation}\label{formgen2}
		\chi(G_0)=\sum_{k=1}^n(-1)^{k-1}\sum_{i=1}^{\binom{n}{k}}\mbox{Vol}_\Z(\Gamma^{\Delta^{(k)}_i}).   
	\end{equation}    

Let us compute $\mbox{Vol}_\Z(\Gamma^{\Delta^{(k)}_i})$ for each $1\leq k\leq n$. Firstly, since $\Gamma^{\Delta^{(1)}_i}$ are segments joining the origin in $\R^n$ to  $d_1e_1$ or $2d_ie_i$, $2\leq i\leq n$, then
\[
\mbox{Vol}_\Z(\Gamma^{\Delta^{(1)}_1})=d_1 \ \ \text{and} \ \ \mbox{Vol}_\Z(\Gamma^{\Delta^{(1)}_i})=2d_i, \ \ 2\leq i \leq n.\]

Now let $2\leq k\leq n$. A $k$-dimensional face of $\R^n_{\geq0}$ can be generated by a set of the form
	\begin{equation}\label{tipo11}
		\{e_1,e_{i_2},\dots,e_{i_k}\}, \mbox{ for } 2\leq i_2<\dots<i_k\leq  n
	\end{equation}
or
	\begin{equation}\label{tipo21}
		\{e_{j_1},\dots,e_{j_k}\}, \mbox{ for } 2\leq j_1<\dots<j_k\leq n.
	\end{equation}
Let $\Delta^{(k)}_\tau$ and $\Delta^{(k)}_{\tau'}$ be the faces generated by vectors of the form (\ref{tipo11}) and (\ref{tipo21}), respectively. Then, $\mbox{Vol}_\Z(\Gamma^{\Delta^{(k)}_\tau})=2^{k-1}d_1d_{i_2}\dots d_{i_k}$ and $\mbox{Vol}_\Z(\Gamma^{\Delta^{(k)}_{\tau'}}) = 2^{k}d_{j_1}\dots d_{j_{k}}=2\cdot 2^{k-1}d_{j_1}\dots d_{j_{k}}.$

Substituting these volumes in (\ref{formgen2}), we have
\begin{align*}
\chi(G_0)=d_1 + 2 \cdot \sum_{i=2}^{n} d_i &+\sum_{k=2}^{n}(-1)^{k-1}2^{k-1}\sum_{2\leq i_2<\dots <i_k\leq n}^{}d_{1} d_{i_2} \dots d_{i_k}\\
&+2\sum_{k=2}^{n}(-1)^{k-1}2^{k-1}\sum_{2\leq j_1<\dots <j_k\leq n}^{}d_{j_1} \dots d_{j_k}.
\end{align*}
This can be rewritten as
\begin{align*}
\chi(G_0)= \Big(d_1 + \sum_{i=2}^{n} d_i &+\sum_{k=2}^{n}(-1)^{k-1}2^{k-1}\sum_{2\leq i_2<\dots <i_k\leq n}^{}d_{1} d_{i_2} \dots d_{i_k}\\
&+\sum_{k=2}^{n}(-1)^{k-1}2^{k-1}\sum_{2\leq j_1<\dots <j_k\leq n}^{}d_{j_1} \dots d_{j_k}\Big)\\ 
&+ \sum_{i=2}^{n} d_i +  \sum_{k=2}^{n}(-1)^{k-1}2^{k-1}\sum_{2\leq j_1<\dots <j_k\leq n}^{}d_{j_1} \dots d_{j_k}.
\end{align*}
Thus, Theorem \ref{differentds} implies
$$\chi(G_0)= \chi(F_0)+ \sum_{i=2}^{n} d_i +  \sum_{k=2}^{n}(-1)^{k-1}2^{k-1}\sum_{2\leq j_1<\dots <j_k\leq n}^{}d_{j_1} \dots d_{j_k}.
$$ 
On the other hand, it is well-known that 
$$
\chi(G_0) = 1 + (-1)^{n-1} \mu(g).
$$ 
Therefore,
$$\chi(F_0)=  1 + (-1)^{n-1} \mu(g) - \big( \sum_{i=2}^{n} d_i +  \sum_{k=2}^{n}(-1)^{k-1}2^{k-1}\sum_{2\leq j_2<\dots <j_k\leq n}^{}d_{j_2} \dots d_{j_k}\big).
$$ 
Now, since $S_n^2$ has an isolated singularity at the origin, by applying \purple{\cite[Theorem 3.1]{BMPS}}, we obtain
\[
Eu_{f,S_n^2}(0)= Eu_{S_n^2}(0) - \chi(F_0).\]
The result follows from Corollary \ref{corobs}.
\end{proof}

\section*{Acknowledgements}

The first and third authors would like to thank for the great hospitality received from Centro de Ciencias Matem\'aticas, UNAM Campus Morelia, during the visit in which this project started.

The authors express their heartfelt gratitude, in loving memory, to Professor Maria Aparecida Soares Ruas for being the generous bridge that made this collaboration possible and for her unwavering encouragement of our work. Her contagious enthusiasm and passion for knowledge opened the doors to the fascinating world of determinantal varieties, which she introduced to us with brilliance and sensitivity. Her presence remains alive in every discovery and every step we take along this scientific journey.

\bibliography{ref}

\vspace{.5cm}
\noindent{\footnotesize \textsc {T. M. Dalbelo, Universidade Federal de Sao Carlos.} \\
Email: thaisdalbelo@ufscar.br}\\
\noindent{\footnotesize \textsc {D. Duarte, Centro de Ciencias Matem\'aticas, UNAM.} \\
E-mail: adduarte@matmor.unam.mx}\\
\noindent{\footnotesize \textsc {D. da N\'obrega Santos, Universidade Federal Rural de Pernambuco.} \\
Email: danilo.nobrega@ufrpe.br}

\end{document}